\documentclass[12pt]{article}

\usepackage{amsfonts}
\newtheorem{thm}{Theorem}[section]

\newtheorem{cor}[thm]{Corollary}
\newtheorem{definition}[thm]{Definition}
\newtheorem{prop}[thm]{Proposition}

\newcommand{\proof
}{\par\medskip\noindent {\bf Proof.\ \ }}
\newtheorem{rem}[thm]{Remark}
\newcommand{\be}{\begin{equation}}
\newcommand{\ee}{\end{equation}}
\newcommand{\openbox}{\leavevmode
  \hbox to8pt{\hfil\vrule\vbox to6pt{\hrule width6pt\vfil\hrule}\vrule}}

\newtheorem{conj}[thm]{Conjecture}
\newcommand{\qed}{\hbox to5pt{ } \hfill \openbox\bigskip\medskip}

\newcommand{\cF}{\mbox{$\cal F$}}

\newcommand{\floor}[1]{\left\lfloor{#1}\right\rfloor}
\newcommand{\vol}[1]{\mathrm{vol}\!\left(#1\right)}

\newcommand{\N}{\mathbb N}
\newcommand{\Z}{\mathbb Z}

\newcommand{\R}{\mathbb R}

\newcommand{\abs}[1]{\left\vert{#1}\right\vert}

\renewcommand{\det}[1]{\mathrm{det}\!\left(#1\right)}

\newcommand{\lcm}[1]{\mathrm{lcm}\!\left\{#1\right\}}

\newcommand{\cone}[1]{\mathrm{cone}\!\left(#1\right)}
\begin{document}

\title{Bounds for smooth Fano polytopes}
\author{G\'abor Heged\"{u}s
\\{\normalsize  \'Obuda University}
\\{\normalsize   John von Neumann Faculty of Informatics}
\\{\normalsize B\'ecsi \'ut 96, Budapest, Hungary, H-1037}
\\{\normalsize hegedus.gabor@nik.uni-obuda.hu}
}
\maketitle

\begin{abstract}
We characterize smooth Fano polytopes in terms of their delta-vector and f-vector. 

As an application we prove that the  delta-vector of a  smooth Fano polytope is unimodal and we give upper  and lower bound on the volume of  smooth Fano polytopes. 
\end{abstract}

\section{Introduction}
Throughout the paper $P$ will be a $d$-dimensional convex lattice polytope in $\R^d$. Denote the boundary of $P$ by $\partial P$.  

Let $L_P(m):=\abs{mP\cap\Z^d}$ denote the number of lattice points in $P$ dilated by a factor of $m\in\Z_{\geq 0}$. Similarly, let $L_{\partial P}(m):=\abs{\partial(mP)\cap\Z^d}$ denote the number of lattice points on the boundary of $mP$. In general the function $L_P$ is a polynomial of degree $d$, called the \emph{Ehrhart polynomial}.

The boundary volume $\vol{\partial P}$ is the sum of the volume of each facet of $P$ normalized with respect to the sublattice containing that facet. For example, in two dimensions the boundary volume  is equal to  the number of lattice points on the boundary of $P$.

In~\cite{Stan80} Stanley proved that the generating function for $L_P$ can be written as a rational function
$$\mathrm{Ehr}_P(t):=\sum_{m\geq 0}L_P(m)t^m=\frac{\delta_0+\delta_1t+\ldots+\delta_dt^d}{(1-t)^{d+1}},$$
where the coefficients $\delta_i$ are non-negative (although possibly zero). If we write $\delta_P(t)=\delta_0+\delta_1 t+\ldots +\delta_d t^d$, then the sequence $\left(\delta_0,\delta_1,\ldots,\delta_d\right)$ is known as the \emph{$\delta$-vector} of $P$, and satisfies some special numerical conditions. For example,
\begin{equation}\label{eq:volume_from_delta}
d!\,\vol{P}=\sum_{k=0}^d\delta_k.
\end{equation}

\begin{definition}\label{def:reflexive}
A convex lattice polytope $P$ is called \emph{reflexive} if the dual polytope
$$
P^\vee:=\{u\in\R^d\mid \left<u,v\right>\le 1\mbox{ for all }v\in P\}
$$
is also a lattice polytope.
\end{definition}

Reflexive polytopes  are in bijective correspondence with  Gorenstein toric Fano varieties (see~\cite{Bat94}). There are many interesting characterizations of reflexive polytopes (see  the list in~\cite{HM04}).

\begin{thm}\label{thm:Reflexive_conditions}
Let $P$ be a $d$-dimensional convex lattice polytope with $0\in P^\circ$, where $P^\circ$ denote the strict interior of $P$. The following statements are equivalent:
\begin{itemize}
\item[(i)] $P$ is reflexive;
\item[(ii)] $d\,\vol{P}=\vol{\partial P}$;
\item[(iii)] $\delta_i=\delta_{d-i}$ for all $0\le i\le d$.
\end{itemize}
\end{thm}

Theorem~\ref{thm:Reflexive_conditions}~(iii) is  known as \emph{Hibi's Palindromic Theorem}~\cite{Hib91}.

\begin{definition}\label{def:smooth}
A $d$-dimensional convex lattice polytope $P$ is called \emph{smooth Fano} if the vertices of any facet of $P$ form a $\Z$-basis of the ambient lattice $\Z^d$.
\end{definition}

Clearly any smooth Fano polytope is simplicial and reflexive. Smooth Fano polytopes correspond  to non-singular toric Fano varieties, and have been classified up to dimension eight~\cite{Obr07}. 
These polytopes are   the subject of much study (for example,~\cite{Baty91,Obr07}).

We recall here the definition of triangulation.
\begin{definition}\label{def:unimodular}
A \em{triangulation}  $\Delta$ of a polytope $P$ is a partition of $P$ into simplices
such that (i) the union of all them equals $P$ and (ii) the intersection of any pair of
them is a  common face.
\end{definition}



We say that a polytope $P$ has the {\em integer decomposition property}, shortly
IDP,
if for each $c\in \N$, $z\in cP\cap {\Z}^d$ there 
exist $x_1, \ldots ,x_c\in P\cap {\Z}^d$ such that $\sum_i x_i=z$.
It is a well-known fact that if  $P$ is a convex lattice polytope with uni-modular triangulation then $P$ has IDP (see \cite[1.2.5]{HPPS21}).



The main motivation of our research was to give appropriate lower and upper bounds on the volume of smooth Fano polytopes in terms of the dimension of the polytope.

First we characterize smooth Fano polytopes in terms of their delta-vector and f-vector.

\begin{thm}\label{prop:smooth_equivalences}
Let $P$ be a $d$-dimensional reflexive simplicial polytope. The following are equivalent:
\begin{itemize}
\item [(i)] $P$ is smooth Fano;
\item [(ii)] $h_k=\delta_k$, for $0\leq k\leq d$;
\item [(iii)] $f_{d-1}=d!\,\vol{P}$;
\item [(iv)] $f_{d-1}=(d-1)!\,\vol{\partial P}$.
\end{itemize}
\end{thm}

\begin{cor} \label{unimodal}
Let $P$ be a $d$-dimensional smooth Fano polytope. 
Then $\delta_i(P)\leq {n-d+i-1 \choose i}$ for each $0\leq i\leq \floor{\frac{d}{2}}$. 
The delta-vector of a smooth Fano polytope is unimodal.
\end{cor}

\begin{rem}

In \cite{APPS22} K. A. Adiprasito, S. Papadakis, V. Petrotou and J. K. Steinmeyer proved that the  delta-vector of a  reflexive polytope  with IDP is unimodal. 

\end{rem}

\begin{cor} \label{smooth_bounds}
Let $P$ be a $d$-dimensional smooth Fano polytope with $n:=\abs{\partial P\cap\Z^d}$. Then
$$
(d-1)n-(d+1)(d-2)\leq d!\,\vol{P}\leq {n-\floor{(d+1)/2}\choose n-d}+{n-\floor{(d+2)/2} \choose n-d}.
$$
\end{cor}

\begin{cor} \label{volume_bound}
Let $P$ be a $d$-dimensional smooth Fano polytope. 
Then
$$
d!\vol{P}\leq 2{\floor{\frac{5}{2}d} \choose 2d}\approx 2\left( \frac{5^5}{4^4} \right)^{d/2}.
$$
\end{cor}



We collected some basic facts about the combinatorics of polytopes in Section 2. We investigate the boundary triangulations of  reflexive polytopes in Section 3. 
We prove our main results in Section 4. In Section 5 we present a conjecture about the upper bound  for the volume of smooth Fano  polytopes.

\section{Combinatorics of polytopes}

Let $d\geq 2$ and $n\geq d+1$ be integers. Consider the convex hull of any $n$ distinct points on the moment curve $\{(t,t^2,\ldots ,t^d):t\in \R\}$. It can be shown that the combinatorial structure of the simplicial $d$--polytope $C(n,d)$ is independent of the actual choice of the points, and this polytope is the cyclic $d$-polytope with $n$ vertices. It can be shown that
\begin{thm} \label{f_vector_cyclic}
$$
f_{d-1}(C(n,d))={n-\floor{(d+1)/2}\choose n-d}+{n-\floor{(d+2)/2} \choose n-d}.
$$
\end{thm}
 In \cite{MM70} Mcmullen proved  the famous Upper Bound Theorem.
\begin{thm} \label{upper_bound_theorem}
(Upper Bound Theorem, Mcmullen, 1970)  
Let $P$ be a $d$-polytope with $n$ vertices. Then
$$
f_j(P)\leq f_j(C(n,d))
$$ 
for all $j$, $1\leq j\leq d-1$.
\end{thm}
Mcmullen based his solution of the Upper Bound Theorem on the following result.
\begin{thm} \label{upper_bound_theorem_h}(Mcmullen, \cite[Lemma~2]{MM70})
Let $P$ be a simplicial $d$-polytope with $n$ vertices. Then
$$
h_i(P)\leq {n-d+i-1\choose i}
$$
for all $i$, $0\leq i\leq d$.
\end{thm}

If we start with a $d$-simplex, one can add
new vertices by building shallow pyramids over facets
to obtain a simplicial convex $d$-polytope with $n$ vertices, called the {\em stacked} polytope. Barnette described the $f$-vector of stacked polytopes and proved the Lower Bound Conjecture in \cite{Ba71} and \cite{Ba73}.
\begin{thm} \label{f_vector_stacked}
Let $P(n,d)$ be a stacked polytope with $n$ vertices. Then
$$
f_k(P(n,d))=\left\{\begin{array}{ll}
\displaystyle{d\choose k}n-{d+1\choose k+1}k,&\mbox{for }1\leq k\leq d-2;\\
(d-1)n-(d+1)(d-2),&\mbox{for }k=d-1.
\end{array}\right.
$$
\end{thm}

\begin{thm} (Lower Bound Theorem, Barnette, 1973)  \label{Lower_Bound_Theorem}
Let $P$ be a simplicial $d$-polytope with $n$ vertices and $P(n,d)$ be a stacked $d$-polytope. 
Then $f_j(P)\geq f_j(P(n,d))$ for all $j$, $1\leq j\leq d-1$.
\end{thm}

The famous $g$-Theorem characterizes completely the $h$-vector of a simplicial polytopes. This result was conjectured by McMullen and proved by Billera and Lee in  \cite{BL80},   \cite{BL81}  and Stanley in \cite{Stan75}. The following Theorem is an immediate consequence of  the $g$-Theorem.
\begin{thm} \label{g_theorem}
If a vector $(h_0,\ldots ,h_d)\in \N^{d+1}$ is the $h$-vector of a simplicial polytope, then
\begin{itemize}
\item[(i)] $h_i=h_{d-i}$ for all $i$, $0\leq i\leq d$ 
\item[(ii)] $h_0=1$, and $h_i\leq h_{i+1}$ for all $i$, $0\leq i\leq d/2-1$.
\end{itemize}
\end{thm}

Theorem \ref{g_theorem} (i) is known as Dehn-Sommerville equations.

\section{Reflexive polytopes}

\begin{definition}\label{def:boundary_triangulation}
Let $P$ be a $d$-dimensional reflexive polytope. Set $V:=\partial P\cap\Z^d$. We say that a simplicial complex $\Delta$ is a \emph{triangulation of $\partial P$ with vertex set $V$} if the following conditions are satisfied:
\begin{itemize}
\item[(i)] If $\sigma,\tau\in\Delta$, then $\sigma\cap\tau$ is a common face of both $\sigma$ and $\tau$;
\item[(ii)] $\bigcup_{\sigma\in\Delta}\sigma=\partial P$.
\end{itemize}
\end{definition}

The boundary triangulation of any smooth Fano polytope is uni-modular, and any reflexive polytope $P$, when dilated by a large enough factor, possess a uni-modular triangulation. More specifically let $\Delta$ be any triangulation of the boundary of $P$, and let $\det{\sigma}$ be the index of the sublattice generated by a simplex $\sigma$ in $\Delta$. If $f:=\lcm{\det{\sigma}\mid\sigma\in\Delta}$, then $fP$ possesses a unimodular triangulation.
The following results are  well-known.
 
\begin{prop}\label{prop:exists_unimodular_delta}
Let $P$ be a $d$-dimensional reflexive polytope. 
The triangulation $\Delta_{\partial P}$ given by the boundary complex of $P$ is a triangulation of $\partial P$ with vertex set $V$.
\end{prop}

\begin{prop}[\protect{\cite[Proposition~2.2]{Hib91}}]\label{prop:Hibi_compressed_delta}
Let $P$ be a $d$-dimensional reflexive polytope. Suppose that $\Delta$ is a triangulation of the boundary of $P$ with the vertex set $\partial P\cap\Z^d$. Then $h_i(\Delta)\leq\delta_i(P)$ for every $0\leq i\leq d$. Moreover, $h_\Delta=\delta_P$ iff $\Delta$ is unimodular.
\end{prop}

As an immediate corollary, one obtains the following relationship between the $h$-vector and 
the $\delta$-vector of a reflexive simplicial polytope:

\begin{cor}\label{cor:reflexive_simplicial_h_delta}
Let $P$ be a $d$-dimensional reflexive simplicial polytope. 
Then $h_i(P)\leq\delta_i(P)$ for every $0\leq i\leq d$.
\end{cor}

\begin{cor} \label{lower_bound_volume}
Let $P$ be a $d$-dimensional reflexive polytope. Then $f_{d-1}(P)\leq d!\vol{P}$.
\end{cor}
\proof 

Let $\Delta$ denote a triangulation of the boundary of $P$ with the vertex set $\partial P\cap\Z^d$. Clearly $f_{d-1}(P)\leq f_{d-1}(\Delta)$. Hence 
by Proposition \ref{prop:Hibi_compressed_delta}
$$
f_{d-1}(\Delta)=\sum_{i=0}^d h_i(\Delta)\leq \sum_{i=0}^d \delta_i(P)= d!\vol{P}.
$$
\qed
\begin{cor} \label{lower_bound_volume2}
Let $P$ be a $d$-dimensional reflexive polytope with $n:=\abs{\partial P\cap\Z^d}$. 
Then $(d-1)n-(d+1)(d-2)\leq d!\vol{P}$.
\end{cor}
\proof
By Corollary \ref{lower_bound_volume} $f_{d-1}(P)\leq d!\vol{P}$. Hence the result follows from Theorem \ref{f_vector_stacked} and Theorem \ref{Lower_Bound_Theorem}. \qed

\section{Proofs of bounds for smooth polytopes}


{\bf Proof of Theorem \ref{prop:smooth_equivalences}:}\\
First we prove the implications (i)~$\Rightarrow$~(ii), (ii)~$\Rightarrow$~(iii) and (iii)~$\Rightarrow$~(i).

(i)~$\Rightarrow$~(ii): 
Let $P$ be a smooth polytope. Consider the triangulation
 $\Delta_{\partial P}$ given by the boundary complex of $P$. 
This is unimodular, hence $h_P=h_{\Delta_{\partial P}}=\delta_P$ by Proposition \ref{prop:Hibi_compressed_delta}.\\

(ii)~$\Rightarrow$~(iii): Suppose that $h_P=\delta_P$. Then, by equation~\ref{eq:volume_from_delta},
$$f_{d-1}=\sum_{i=0}^d h_i=\sum_{i=0}^d\delta_i=d!\,\vol{P}.$$

(iii)~$\Rightarrow$~(i):
Since $P$ is reflexive, by Theorem~\ref{thm:Reflexive_conditions}~(ii) we have that
$$f_{d-1}=d!\,\vol{P}=(d-1)!\,\vol{\partial P}=(d-1)!\,\sum_F\vol{F},$$
where $F$ ranges over all facets of $P$. However, since $P$ is simplicial and integral, $(d-1)!\,\vol{F}=\det{F}\ge 1$, where $\det{F}$ is the index of the sublattice generated by the vertices of $F$. Thus we see that $\det{F}=1$ for all $f_{d-1}$ facets of $P$, and so $P$ is smooth. 

(iii)~$\Leftrightarrow$~(iv): This follows easily from Theorem \ref{thm:Reflexive_conditions} (ii). 
\qed

{\bf Proof of  Corollary \ref{unimodal}:}\\

Corollary \ref{unimodal} follows from Theorem \ref{upper_bound_theorem_h}, Theorem \ref{g_theorem}~(i) and (ii) and Theorem \ref{prop:smooth_equivalences}.
\qed

{\bf Proof of  Corollary \ref{smooth_bounds}:}\\ 

The lower bound follows from  Corollary \ref{lower_bound_volume2}. The upper bound follows from Theorem \ref{f_vector_cyclic}, Theorem \ref{upper_bound_theorem}  and Theorem \ref{prop:smooth_equivalences}.\qed

{\bf Proof of  Corollary \ref{volume_bound}:}\\ 
Let $P$ be a $d$-dimensional smooth Fano polytope. Let $n:=\abs{\partial P\cap\Z^d}$.
Casagrande  proved in  \cite{Cas04} that
$n\leq 3d$. Hence using that
$$
{n-\floor {(d+1)/2}\choose d-\floor {(d+1)/2}}
$$
is a monotone increasing function of $n$, it follows from 
Corollary \ref{smooth_bounds} that
$$
d!\vol{P}\leq {3d-\floor{(d+1)/2}\choose 2d}+{3d-\floor{(d+2)/2} \choose 2d}\leq 2{\floor{\frac{5}{2}d} \choose 2d}.
$$
To approximate the binomial coefficient ${\frac{5}{2}d \choose 2d}$ we used the following bounds of Sondow et al.  (see  \cite{S05}, \cite{SZ06}):  
$$
\frac{1}{4rs}\Big( \frac{(r+1)^{r+1}}{r^r} \Big)^s\leq {(r+1)s\choose s}\leq \Big( \frac{(r+1)^{r+1}}{r^r} \Big)^s,
$$
where $s\geq 1$ is a positive integer and $r\in \R$ is an arbitrary real.
\qed

\section{Concluding remarks}

We conjecture the following upper bound for the volume of smooth Fano lattice polytopes, which is a sharpening of the upper bound appearing in Corollary  \ref{smooth_bounds}.

Our conjecture based on the classification of smooth Fano lattice polytopes in ~\cite{Obr07}. 
\begin{conj}\label{conj:smooth}
Let $P$ be a $d$-dimensional smooth Fano polytope. 
Then the following upper bound is sharp:
$$
d!\vol{P}\leq \frac 12 (3-(-1)^d) 6^{\lfloor d/2\rfloor}.
$$
If $d$ is odd, then there exist precisely two unimodularly equvivalent   $d$-dimensional smooth Fano polytopes  $P$ with  $d!\vol{P}=\frac 12 (3-(-1)^d) 6^{\lfloor d/2\rfloor}$. For both polytopes $f_0=\frac 12(6d+(-1)^{d}-1)$ and only one of these polytopes is centrally symmetric.

If $d$ is even, then there exists precisely one   $d$-dimensional smooth Fano polytope  $P$ with  $d!\vol{P}= (3-(-1)^d) 6^{\lfloor d/2\rfloor}$. Then $f_0=\frac 12(6d+(-1)^{d}-1)$ and $P$ is centrally symmetric.
\end{conj}
\end{document}